\documentclass{birkjour}
\usepackage{amsmath,amssymb,amsthm}

\theoremstyle{definition}
\newtheorem{theorem}{Theorem} 
\newtheorem{definition}{Definition}[section]
\newtheorem{lemma}{Lemma} 

\newcommand{\bell}{\textup{B}}

\begin{document}

\title[Sum of inverses of odd divisors]
 {An identity for the sum of inverses of odd divisors of $n$ in terms of the number of representations of $n$ as a sum of $r$ squares}

\author[S. K. Jha]{Sumit Kumar Jha}

\address{%
International Institute of Information Technology\\
Hyderabad-500 032, India}

\email{kumarjha.sumit@research.iiit.ac.in}





\begin{abstract}
Let $$\sum_{\substack{d|n\\ d\equiv 1 (2)}}\frac{1}{d}$$ denote the sum of inverses of odd divisors of a positive integer $n$, and let $c_{r}(n)$ be the number of representations of $n$ as a sum of $r$ squares where representations with different orders and different signs are counted as distinct. The aim is of this note is to prove the following interesting combinatorial identity:
$$
\sum_{\substack{d|n\\ d\equiv 1 (2)}}\frac{1}{d}=\frac{1}{2}\,\sum_{r=1}^{n}\frac{(-1)^{n+r}}{r}\,\binom{n}{r}\, c_{r}(n).
$$
\end{abstract}

\maketitle

\section{Main result}
In the following, let  $$\sum_{\substack{d|n\\ d\equiv 1 (2)}}\frac{1}{d}$$ denote the sum of inverses of odd divisors of a positive integer $n$. 
\begin{definition}\cite[formula 7.324]{Fine}
Let $\theta(q)$ be the following infinite product 
$$\theta(q):=\prod_{j=1}^{\infty}\frac{1-q^{j}}{1+q^{j}}=\sum_{n=-\infty}^{\infty}(-1)^{n}q^{n^{2}}$$ where $|q|<1$.
\end{definition}
\begin{definition}
For any positive integer $r$ define $c_{r}(n)$ by  
$$\theta(q)^{r}=\sum_{n=0}^{\infty}c_{r}(n)(-1)^{n}q^{n}$$
where $c_{r}(n)$ be the number of representations of $n$ as a sum of $r$ squares where representations with different orders and different signs are counted as distinct.
\end{definition}
Our aim is to derive the following identity.
\begin{theorem}
\label{main}
For all positive integers $n$ we have
\begin{equation}
\label{maineq}
\sum_{\substack{d|n\\ d\equiv 1 (2)}}\frac{1}{d}=\frac{1}{2}\sum_{r=1}^{n}\frac{(-1)^{n+r}}{r}\,\binom{n}{r}\, c_{r}(n).
\end{equation}
\end{theorem}
We require following two lemmas for our proof.
\begin{lemma}
For all positive integers $n$ we have
\begin{equation}
\label{lem1}
2\, \sum_{\substack{d|n\\ d\equiv 1 (2)}}\frac{1}{d}=\frac{1}{n!}\sum_{k=1}^n (-1)^{k}\, (k-1)!\, B_{n,k}\left(\theta'(0),\theta''(0),\dots,\theta^{(n-k+1)}(0)\right)
\end{equation}
where $B_{n,k}\equiv\bell_{n,k}(x_1,x_2,\dotsc,x_{n-k+1})$ are the partial Bell polynomials defined by \cite[p. 134]{Comtet}
\begin{equation*}
\bell_{n,k}(x_1,x_2,\dotsc,x_{n-k+1})=\sum_{\substack{1\le i\le n,\ell_i\in\mathbb{N}\\ \sum_{i=1}^ni\ell_i=n\\ \sum_{i=1}^n\ell_i=k}}\frac{n!}{\prod_{i=1}^{n-k+1}\ell_i!} \prod_{i=1}^{n-k+1}\Bigl(\frac{x_i}{i!}\Bigr)^{\ell_i}.
\end{equation*}
\end{lemma}
\begin{proof}
It is easy to see that
\begin{align*}
\log(\theta(q))&=\sum_{j=1}^{\infty}\log(1-q^{j})-\sum_{j=1}^{\infty}\log(1+q^{j})\\
&=-\sum_{j=1}^{\infty}\sum_{l=1}^{\infty}\frac{q^{lj}}{l}+\sum_{j'=1}^{\infty}\sum_{l'=1}^{\infty}\frac{q^{l'j'}(-1)^{l'}}{l'}\\
&=-\sum_{n=1}^{\infty}q^{n}\left(\sum_{d|n}\frac{1-(-1)^{d}}{d}\right).
\end{align*}
Let $f(q)=\log{q}$. Using Fa\`{a} di Bruno's formula \cite[p. 137]{Comtet} we have
\begin{equation}
\label{faa}
{d^n \over dq^n} f(\theta(q)) = \sum_{k=1}^n f^{(k)}(\theta(q))\cdot B_{n,k}\left(\theta'(q),\theta''(q),\dots,\theta^{(n-k+1)}(q)\right).
\end{equation}
Since $f^{(k)}(q)=\frac{(-1)^{k-1}\,(k-1)!}{q^{k}}$ and $\theta(0)=1$, letting $q\rightarrow 0$ in the above equation gives us Equation \eqref{lem1}.
\end{proof}
\begin{lemma}
We have, for positive integers $n,k$,
\begin{equation}
\label{lem2}
B_{n,k}\left(\theta'(0),\theta''(0),\dots,\theta^{(n-k+1)}(0)\right)=(-1)^{n}\,\frac{n!}{k!}\sum_{r=1}^{k}(-1)^{k-r}\binom{k}{r}c_{r}(n)
\end{equation}
\end{lemma}
\begin{proof}
We start with the generating function for the partial Bell polynomials \cite[Equation (3a') on p. 133]{Comtet}, and proceed as follows
\begin{align*}
{\displaystyle \sum _{n=k}^{\infty }B_{n,k}\left(\theta'(0),\theta''(0),\dots,\theta^{(n-k+1)}(0)\right){\frac {q^{n}}{n!}}}
&= {\frac {1}{k!}}\left(\sum _{j=1}^{\infty }\theta^{(j)}(0){\frac {q^{j}}{j!}}\right)^{k} \\
&=\frac{1}{k!}(\theta(q)-1)^{k}\\
&=\frac{1}{k!}\sum_{r=0}^{k}(-1)^{k-r}\binom{k}{r}\theta(q)^{r}\\
&=\frac{1}{k!}\sum_{r=0}^{k}(-1)^{k-r}\binom{k}{r}\sum_{n=0}^{\infty}(-1)^{n}c_{r}(n)q^{n}
\end{align*}
to conclude Equation \eqref{lem2}.
\end{proof}
\begin{proof}[Proof of Theorem \ref{main}]
We combine Equation \eqref{lem1} and Equation \eqref{lem2} to obtain
\begin{align*}
2\, \sum_{\substack{d|n\\ d\equiv 1 (2)}}\frac{1}{d}&=(-1)^{n}\sum_{k=1}^{n}\frac{1}{k}\sum_{r=1}^{k}(-1)^{r}\, \binom{k}{r}\,c_{r}(n)\\
&=(-1)^{n}\sum_{r=1}^{n}(-1)^{r}c_{r}(n)\sum_{k=r}^{n}\frac{1}{k}\binom{k}{r}\\
&=(-1)^{n}\sum_{r=1}^{n}\frac{(-1)^{r}}{r}\,\binom{n}{r}\,c_{r}(n),
\end{align*}
where we used the identity $\sum_{k=r}^{n}\frac{1}{k}\binom{k}{r}=\frac{1}{r}\binom{n}{r}$ which can be proved using the Pascal's formula
$$
\binom{k}{r-1}=\binom{k+1}{r}-\binom{k}{r}.
$$
This concludes proof of our main result Equation \eqref{maineq}.
\end{proof}

\end{document}